\documentclass{amsart}
\usepackage{mathptmx, amssymb, enumerate, colonequals, amscd, slashbox, tabularx, bm, url, float, epsfig, psfrag}
\usepackage{color}
\usepackage[colorlinks=true, pagebackref, hyperindex, citecolor=green, linkcolor=red]{hyperref}
\restylefloat{figure}

\newtheorem{theorem}{Theorem}[section]

\newtheorem{corollary}[theorem]{Corollary}

\theoremstyle{definition}
\newtheorem{example}[theorem]{Example}
\newtheorem{question}[theorem]{Question}
\numberwithin{equation}{theorem}

\def\ge{\geqslant}
\def\le{\leqslant}
\def\phi{\varphi}
\def\tilde{\widetilde}
\def\to{\longrightarrow}
\def\mapsto{\longmapsto}
\def\onto{\relbar\joinrel\twoheadrightarrow}
\def\ker{\operatorname{ker}}
\def\rank{\operatorname{rank}}
\def\Hom{\operatorname{Hom}}
\def\Ext{\operatorname{Ext}}
\def\Spec{\operatorname{Spec}}
\def\edim{\operatorname{edim}}

\def\bsi{{\boldsymbol{i}}}
\def\bsx{{\boldsymbol{x}}}
\def\bsy{{\boldsymbol{y}}}
\def\bsz{{\boldsymbol{z}}}
\def\bzero{{\boldsymbol{0}}}

\def\fraka{\mathfrak{a}}
\def\frakm{\mathfrak{m}}
\def\frakp{\mathfrak{p}}

\def\CC{\mathbb{C}}
\def\NN{\mathbb{N}}
\def\PP{\mathbb{P}}
\def\RR{\mathbb{R}}
\def\ZZ{\mathbb{Z}}

\newcolumntype{C}[1]{>{\centering\arraybackslash}p{#1}}

\begin{document}
\title{Koszul and local cohomology, and a question of Dutta}

\author{Linquan Ma}
\address{Department of Mathematics, Purdue University, 150 N University St., West Lafayette, IN~47907, USA}
\email{ma326@purdue.edu}

\author{Anurag K. Singh}
\address{Department of Mathematics, University of Utah, 155 South 1400 East, Salt Lake City, UT~84112, USA}
\email{singh@math.utah.edu}

\author{Uli Walther}
\address{Department of Mathematics, Purdue University, 150 N University St., West Lafayette, IN~47907, USA}
\email{walther@math.purdue.edu}

\thanks{L.M. was supported by NSF Grant DMS~1901672, NSF FRG Grant DMS~1952366, and by a fellowship from the Sloan Foundation, A.K.S.~by NSF grant DMS~1801285, and U.W.~by the Simons Foundation Collaboration Grant for Mathematicians~580839. A.K.S. thanks Purdue University and his coauthors for their hospitality. The authors are grateful to Srikanth B.~Iyengar and to the referee for several helpful~comments.}

\begin{abstract}
For $(A,\frakm)$ a local ring, we study the natural map from the Koszul cohomology module $H^{\dim A}(\frakm;\,A)$ to the local cohomology module $H^{\dim A}_\frakm(A)$. We prove that the injectivity of this map characterizes the Cohen-Macaulay property of the ring $A$. We also answer a question of Dutta by constructing normal rings $A$ for which this map is zero.
\end{abstract}
\maketitle

\section{Introduction}

For a commutative Noetherian local ring $(A,\frakm)$, we study the natural map from the Koszul cohomology module $H^{\dim A}(\frakm;\,A)$ to the local cohomology module $H^{\dim A}_\frakm(A)$, and use this to answer a question raised by Dutta~\cite{Dutta:LMS}, Question~\ref{question:dutta} below. The motivation for Dutta's question stems from Hochster's \emph{monomial conjecture}~\cite[page~33]{Hochster:Nagoya} that occupies a central place in local algebra; this is the conjecture that if $\bsz\colonequals z_1,\dots,z_n$ form a system of parameters for a local ring~$A$, then for each $t\in\NN$ one has
\[
(z_1\cdots z_n)^t\ \notin\ (z_1^{t+1},\dots,z_n^{t+1})A.
\]
An equivalent formulation of the conjecture is that the natural map
\[
\phi^n_\bsz \colon H^n(\bsz;\,A)\to H^n_{\bsz A}(A),
\]
as discussed in \S\ref{section:koszul}, is nonzero. The monomial conjecture was proved for rings containing a field by Hochster, in the same paper where it was first formulated. The case of rings of dimension at most two is straightforward; for decades, the conjecture remained unresolved for mixed characteristic rings of dimension greater than or equal to three, as did its equivalent formulations, \emph{the direct summand conjecture}, \emph{the canonical element conjecture}, and the \emph{improved new intersection conjecture}. In \cite{Heitmann} Heitmann proved these equivalent conjectures for mixed characteristic rings of dimension three; more recently, Andr\'e~\cite{Andre} settled the mixed characteristic case in full generality, with Bhatt~\cite{Bhatt} establishing a derived variant. Related homological conjectures including \emph{Auslander's zerodivisor conjecture} and \emph{Bass's conjecture} had been settled earlier by Roberts \cite{Roberts:intersection}.

For the setup of Dutta's question, let $A$ be a complete local ring. Using the Cohen structure theorem, $A$ can be written as the homomorphic image of a complete regular local ring; this surjection may be factored so as to obtain a complete Gorenstein local ring $R$, such that $A$ is the homomorphic image of $R$, and $\dim R=\dim A$. Dutta~\cite[page~50]{Dutta:LMS} asked:

\begin{question}
\label{question:dutta}
Let $A$ be a complete normal local ring. Let $(R,\frakm)$ be a Gorenstein ring with a surjective homomorphism $R\onto A$, such that $n\colonequals\dim R=\dim A$. Does the natural map
\begin{equation}
\label{equation:question:dutta}
\Ext^n_R(R/\frakm,\,A)\to H^n_\frakm(A)
\end{equation}
have a nonzero image?
\end{question}

We prove that the answer to the above is negative in the following strong sense: we construct a complete normal local ring $A$ such that for \emph{each} Gorenstein ring $R$ with $R\onto A$ and $\dim R=\dim A$, the map~\eqref{equation:question:dutta} is zero. Our approach is via studying a related question on maps from Koszul cohomology to local cohomology: Let $\bsz\colonequals z_1,\dots,z_t$ be elements of $R$ that generate an $\frakm$-primary ideal. Let $P_\bullet$ be a projective resolution of $R/\frakm$ as an~$R$-module. The canonical surjection $R/\bsz R\onto R/\frakm$ lifts to a map of complexes
\[
K_\bullet(\bsz;\,R)\to P_\bullet,
\]
where $K_\bullet(\bsz;\,R)$ denotes the homological Koszul complex. Applying $\Hom_R(-,\,A)$ to the above and taking cohomology, one obtains the map
\begin{equation}
\label{equation:ext:koszul}
\Ext^n_R(R/\frakm,\,A)\to H^n(\bsz;\,A),
\end{equation}
where $H^n(\bsz;\,A)$ denotes Koszul cohomology. The map~\eqref{equation:question:dutta} factors as a composition of~\eqref{equation:ext:koszul} and the natural map from Koszul cohomology to local cohomology
\begin{equation}
\label{equation:koszul:map}
H^n(\bsz;\,A)\to H^n_\frakm(A);
\end{equation}
the map above is described explicitly in~\S\ref{section:koszul}. Note that in~\eqref{equation:koszul:map}, the ring $R$ no longer plays a role: the elements $\bsz$ may be replaced by their images in $A$; likewise, the maximal ideal of $R$ may be replaced by that of $A$. In Theorem~\ref{theorem:segre} we construct normal graded rings $A$ for which the map~\eqref{equation:koszul:map} is zero; localizing at the homogeneous maximal ideal and taking the completion, one obtains examples where the answer to Question~\ref{question:dutta} is negative.

Quite generally, for $(A,\frakm)$ a local ring and $n\colonequals\dim A$, we prove that the injectivity of the natural map~$H^n(\fraka;\,A)\to H^n_\fraka(A)$ for some (or each) $\frakm$-primary $\fraka$ is equivalent to the ring $A$ being Cohen-Macaulay, Theorem~\ref{theorem:cohen:macaulay}; here, and in the sequel, we use $K^\bullet(\fraka;\,A)$ to denote the cohomological Koszul complex on a \emph{minimal} set of generators for an ideal~$\fraka$, and $H^\bullet(\fraka;\,A)$ for its cohomology. In \S\ref{section:koszul} we record definitions and preliminary material. While~\S\ref{section:injectivity} is largely devoted to the injectivity of the map $H^n(\frakm;\,A)\to H^n_\frakm(A)$,~\S\ref{section:nonvanishing} investigates the nonvanishing and the kernel. Theorem~\ref{theorem:buchsbaum} records a case where we obtain precise information on the kernel of the map $H^n(\frakm;\,A)\to H^n_\frakm(A)$, that we then illustrate with several examples, including some involving Stanley-Reisner rings, \S\ref{section:stanley:reisner}.

All rings under consideration in this paper are Noetherian; by a \emph{local} ring $(A,\frakm)$, we mean a Noetherian ring $A$ with a unique maximal ideal $\frakm$.

\section{Graded Koszul and local cohomology, and limit closure}
\label{section:koszul}

We record some preliminaries on Koszul and local cohomology; the discussion below is in the graded context, in the form used in the proof of Theorems~\ref{theorem:segre}. Ignoring the grading and degree shifts, one has similar statements outside of the graded setting.

Let $A$ be an $\NN$-graded ring, and $z$ a homogeneous ring element. Then one has a degree-preserving map from the Koszul complex $K^\bullet(z;\,A)$ to the \v Cech complex $C^\bullet(z;\,A)$ as below:
\[
\CD
0 @>>> A @>z>> A(\deg z) @>>> 0\phantom{.}\\
@. @V1VV @V\frac{1}{z}VV\\
0 @>>> A @>>> A_z @>>> 0.
\endCD
\]
For a sequence of homogeneous elements $\bsz\colonequals z_1,\dots,z_t$, one similarly has
\[
\CD
K^\bullet(\bsz;\,A)\colonequals\bigotimes_i K^\bullet(z_i;\,A) @>>> \bigotimes_i C^\bullet(z_i;\,A)\equalscolon C^\bullet(\bsz;\,A).
\endCD
\]
For each $m\ge 0$, the induced map from Koszul cohomology to local cohomology modules
\[
\phi^m_\bsz\colon H^m(\bsz;\,A)\to H^m_{\bsz A}(A)
\]
is degree-preserving, and what we refer to as the \emph{natural} map. For homogeneous elements~$\bsz$ and~$w$ in $A$, one has a commutative diagram with degree-preserving maps and exacts rows:
\begin{equation}
\label{equation:koszul:local}
\minCDarrowwidth20pt
\CD
@>>> H^{m-1}(\bsz;\,A) @>{\pm w}>>H^{m-1}(\bsz;\,A)(\deg w) @>>> H^m(\bsz,w;\,A) @>>> H^m(\bsz;\,A) @>>>\\
@. @V{\phi^{m-1}_{\bsz}}VV @V{\pm\frac{1}{w}\phi^{m-1}_{\bsz}}VV @V{\phi^m_{\bsz,w}}VV @V{\phi^m_{\bsz}}VV\\
@>>> H^{m-1}_{(\bsz)}(A) @>>>H^{m-1}_{(\bsz)}(A_w) @>>> H^m_{(\bsz, w)}(A) @>>> H^m_{(\bsz)}(A) @>>>
\endCD
\end{equation}

Set $n\colonequals\dim A$, and fix a homogeneous system of parameters $\bsz\colonequals z_1,\dots,z_n$ for $A$. The map $\phi^n_\bsz \colon H^n(\bsz;\,A)\to H^n_{\bsz A}(A)$ then takes form
\begin{equation}
\label{equation:top}
\frac{A}{\bsz A}(\textstyle\sum\deg z_i)\to \dfrac{A_{z_1\cdots z_n}}{\sum A_{z_1\cdots \hat{z_i}\cdots z_n}},
\qquad
1\mapsto\left[\dfrac{1}{z_1\cdots z_n}\right].
\end{equation}
Following \cite[\S2.5]{Huneke:Smith}, the \emph{limit closure} of the parameter ideal $\bsz A$ is the ideal
\[
(\bsz A)^{\lim}\colonequals\{x\in A\mid xz_1^j\cdots z_n^j\in (z_1^{j+1},\,\dots,\,z_n^{j+1})A \text{\ \ for }j\gg0\}.
\]
Using~\eqref{equation:top}, it is readily seen that $(\bsz A)^{\lim}/\bsz A$ is the kernel of $\phi^n_\bsz \colon H^n(\bsz;\,A)\to H^n_{\bsz A}(A)$. It follows that $(\bsz A)^{\lim}$ is unchanged if one takes a different choice of \emph{minimal} generators for the ideal $\bsz A$. When $\bsz$ is a regular sequence, it is easily checked that $(\bsz A)^{\lim}=\bsz A$.

\begin{example}
\label{example:real}
Consider $A\colonequals\RR[x,y,ix,iy]$, which is a subring of the polynomial ring $\CC[x,y]$. Then $A$ is a standard graded ring, with homogeneous system of parameters $x,y$. Since
\[
ix(xy)=iy(x^2)\ \in\ (x^2,y^2),
\]
one has $ix\in (x,y)^{\lim}$. Similarly, $iy\in (x,y)^{\lim}$, so $(x,y)^{\lim}$ equals the homogeneous maximal ideal $\frakm$ of $A$. The map $H^2(\frakm;\,A)\to H^2_\frakm(A)$ is zero by an argument similar to the one used in the proof of Theorem~\ref{theorem:segre}; alternatively, see Example~\ref{example:two:lines}.
\end{example}

The ring in the example above is not normal; this leads to:

\begin{question}
Does there exist a non-regular complete normal local ring $(A,\frakm)$, with a system of parameters $\bsz$, such that $(\bsz A)^{\lim}=\frakm$?
\end{question}

For a ring $A$ that meets the conditions above, the answer to Question~\ref{question:dutta} is negative: let~$\bsz\colonequals z_1,\dots,z_n$ be a system of parameters with $(\bsz A)^{\lim}=\frakm$, and take $R$ as in Question~\ref{question:dutta}. The map $\Ext^n_R(R/\frakm_R,\,A)\to H^n_\frakm(A)$ factors through $H^n(\bsz;\,A)=A/\bsz A$. As~$\Ext^n_R(R/\frakm_R,\,A)$ is annihilated by $\frakm$, so is its image in $A/\bsz A$. Hence this image is a submodule of
\[
0:_{A/\bsz A}\frakm\ \subseteq\ \frakm/\bsz A,
\]
where the containment above holds since $A$ is not regular. Since $H^n(\bsz;\,A)\to H^n_\frakm(A)$ has kernel $(\bsz A)^{\lim}/\bsz A = \frakm/\bsz A$, it follows that $\Ext^n_R(R/\frakm_R,\,A)\to H^n_\frakm(A)$ must be zero.

\section{Injectivity of the map from Koszul to local cohomology}
\label{section:injectivity}

We begin with a characterization of the Cohen-Macaulay property:

\begin{theorem}
\label{theorem:cohen:macaulay}
Let $(A,\frakm)$ be a local ring; set $n\colonequals\dim A$. Then the following are equivalent:
\begin{enumerate}[\,\rm(1)]
\item The ring $A$ is Cohen-Macaulay.
\item The natural map $H^n(\fraka;\,A)\to H^n_\fraka(A)$ is injective for each $\frakm$-primary ideal $\fraka$ of $A$.
\item The natural map $H^n(\fraka;\,A)\to H^n_\fraka(A)$ is injective for some $\frakm$-primary ideal $\fraka$ of $A$.
\end{enumerate}
\end{theorem}

\begin{proof}
Suppose $\fraka$ is an $\frakm$-primary ideal. We claim that there exists a minimal generating set $z_1,\dots,z_t$ for $\fraka$ such that $z_1,\dots,z_n$ is a system of parameters. For $i<n$, it suffices to choose an element $z_{i+1}$, not in any minimal prime of $(z_1,\dots,z_i)A$, such that
\[
z_{i+1}\in\fraka\setminus\frakm\fraka.
\]
This may be accomplished using the version of prime avoidance where up to two of the ideals need not be prime, see for example~\cite[Theorem~81]{Kaplansky}.

Assume (1). Suppose $\bsz$ generates an $\frakm$-primary ideal, and $w\in\frakm$ is an additional element. Using~\eqref{equation:koszul:local}, the vanishing of~$H^{n-1}(\bsz;\,A)$ implies that the map
\[
H^n(\bsz,w;\,A) \to H^n(\bsz;\,A)
\]
is injective. Thus, if the map $\phi^n_\bsz\colon H^n(\bsz;\,A)\to H^n_{\bsz A}(A)=H^n_\frakm(A)$ is injective, then so is the map $\phi^n_{\bsz,w}\colon H^n(\bsz,w;\,A)\to H^n_{(\bsz,w)A}(A) = H^n_\frakm(A)$. Hence the proof of (2) reduces to the case where $\fraka$ is generated by a system of parameters $\bsz$. But then~$\bsz$ is a regular sequence, and the injectivity follows.

It is immediate that (2) implies (3). Next, assume (3), i.e., that $\fraka$ is an $\frakm$-primary ideal and that $H^n(\fraka;\,A)\to H^n_\fraka(A)$ is injective. We first consider the case where $\fraka$ is generated by a system of parameters $\bsz$. In this case, the fact that $\bsz$ is a regular sequence on $A$ follows from \cite[Corollary~2.4]{CHL}, though one may also argue as follows: The injectivity translates as~$(\bsz A)^{\lim}=\bsz A$. Using~$e(\bsz A)$ to denote the multiplicity of the ideal $\bsz A$, one has
\[
e(\bsz A) \le \ell(A/\bsz A),
\]
with equality holding precisely if $A$ is Cohen-Macaulay, see \cite[Corollary~4.7.11]{Bruns:Herzog}. But
\[
\ell(A/(\bsz A)^{\lim}) \le e(\bsz A)
\]
by \cite[Theorem~9]{MQS}, so $A$ is Cohen-Macaulay.

For the general case, take a minimal generating set $z_1,\dots,z_t$ for $\fraka$ such that $z_1,\dots,z_n$ is a system of parameters. Suppose $n<t$. Using~\eqref{equation:koszul:local}, one has a commutative diagram
\[
\CD
H^n(z_1,\dots,z_t;\,A) @>>> H^n(z_1,\dots,z_{t-1};\,A)\\
@V{\phi^n_\bsz}VV @VVV\\
H^n_\fraka(A) @= H^n_\fraka(A)
\endCD
\]
Since $\phi^n_\bsz$ is injective by assumption, it follows that
\[
H^n(z_1,\dots,z_t;\,A) \to H^n(z_1,\dots,z_{t-1};\,A)
\]
is injective, and hence that multiplication by $z_t$ on $H^{n-1}(z_1,\dots,z_{t-1};\,A)$ is surjective. But then, by Nakayama's lemma, $H^{n-1}(z_1,\dots,z_{t-1};\,A)=0$, so $A$ is Cohen-Macaulay.
\end{proof}

The next theorem provides a large class of rings for which the answer to Question~\ref{question:dutta} is negative; while the rings below are graded, the relevant issues are unchanged under localization and completion.

\begin{theorem}
\label{theorem:segre}
Let $k$ be a field; take polynomial rings $k[x_1,\dots,x_b]$ and~$C\colonequals k[y_1,\dots,y_c]$, where $b\ge 3$ and $c\ge 2$. Let $f(\bsx)$ be a homogeneous polynomial such that the hypersurface
\[
B\colonequals k[x_1,\dots,x_b]/(f(\bsx))
\]
is normal. Then the Segre product $A\colonequals B\,\#\,C$ is a normal ring of dimension $b+c-2$.

Let~$\frakm$ denote the homogeneous maximal ideal of $A$. Then the following are equivalent:
\begin{enumerate}[\,\rm(1)]
\item The ring $A$ is Cohen-Macaulay.
\item The polynomial $f(\bsx)$ has degree less than $b$.
\item The natural map $H^{b+c-2}(\frakm;\,A)\to H^{b+c-2}_\frakm(A)$ is injective.
\item The natural map $H^{b+c-2}(\frakm;\,A)\to H^{b+c-2}_\frakm(A)$ is nonzero.
\end{enumerate}
\end{theorem}

\begin{proof}
Regarding the normality of $A$, note that $B\otimes_k C$ is a polynomial ring over $B$, hence normal; it follows that its pure subring $A$ is normal as well. The dimension of~$A$ and the equivalence of (1) and (2) may be obtained from the K\"unneth formula for local cohomology~\cite[Theorem~4.1.5]{GW}. Note that (1) and (3) are equivalent by Theorem~\ref{theorem:cohen:macaulay}, and that~(3) trivially implies (4).

Set $d\colonequals\deg f(\bsx)$ and assume that $d\ge b$. To complete the proof, we show that the map
\[
H^{b+c-2}(\frakm;\,A)\to H^{b+c-2}_\frakm(A)
\]
is zero. Fix the polynomial ring
\[
S\colonequals k[z_{ij}\mid 1\le i\le b,\ 1\le j\le c]
\]
with the $k$-algebra surjection $\pi\colon S\onto A$, where $z_{ij}\mapsto x_iy_j$. We work with the standard~$\NN$-gradings on~$S$ and $A$, i.e., $\deg z_{ij}=1=\deg x_iy_j$. Note that the minimal generators for~$\ker\pi$ have degree $2$ and degree $d$, with the degree $2$ generators being
\[
z_{ij}z_{rs}-z_{is}z_{rj}.
\]
Specifically,
\begin{equation}
\label{equation:determinantal}
A_t = {[S/I_2(Z)]}_t\qquad \text{for }t\le d-1,
\end{equation}
where $I_2(Z)$ is the ideal generated by the size $2$ minors of the matrix $Z\colonequals(z_{ij})$.

The K\"unneth formula gives
\[
H^{b+c-2}_\frakm(A)=H^{b-1}_{\frakm_B}(B)\,\#\,H^c_{\frakm_C}(C).
\]
Since ${[H^c_{\frakm_C}(C)]}_j=0$ for $j>-c$, it follows that
\[
{[H^{b+c-2}_\frakm(A)]}_j=0\qquad\text{for }j>-c.
\]
The images of $\bsz\colonequals z_{11},\dots,z_{bc}$ are minimal generators for $\frakm$, so it suffices to show that
\[
{[H^{b+c-2}(\bsz;\,A)]}_j=0\qquad\text{for }j\le -c.
\]
Since $\deg z_{ij}=1$ for each $i,j$, the Koszul complex $K^\bullet(\bsz;\,A)$ has the form
\[
\CD
0 @>>> A @>>> \bigoplus A(1) @>>> \bigoplus A(2) @>>> \bigoplus A(3) @>>> \cdots.
\endCD
\]
Fix $j$ with $j\le -c$. The graded strand of the Koszul complex computing ${[H^{b+c-2}(\bsz;\,A)]}_j$ is
\begin{equation}
\label{equation:strand}
\CD
\bigoplus A_{b+c-3+j} @>\alpha>> \bigoplus A_{b+c-2+j} @>\beta>> \bigoplus A_{b+c-1+j},
\endCD
\end{equation}
where the nonzero entries of the matrices for $\alpha$ and $\beta$ are linear forms in the $z_{ij}$. The condition $j\le -c$ implies that $b+c-1+j\le b-1\le d-1$. In light of~\eqref{equation:determinantal}, it follows that the cohomology of~\eqref{equation:strand} coincides with that of
\[
\CD
\bigoplus {[S/I_2(Z)]}_{b+c-3+j} @>\alpha>> \bigoplus {[S/I_2(Z)]}_{b+c-2+j} @>\beta>> \bigoplus {[S/I_2(Z)]}_{b+c-1+j}.
\endCD
\]
But the Koszul cohomology module ${H^{b+c-2}(\bsz;\,S/I_2(Z))}$ is zero since $S/I_2(Z)$ is a Cohen-Macaulay ring of dimension $b+c-1$.
\end{proof}

\section{Nonvanishing of the map from Koszul to local cohomology}
\label{section:nonvanishing}

Let $(A,\frakm)$ be a local ring; set $n\colonequals\dim A$. Theorem~\ref{theorem:cohen:macaulay} characterizes the injectivity of the map $\phi^n_\frakm\colon H^n(\frakm;\,A)\to H^n_\frakm(A)$. We next discuss when this map is nonzero.

A \emph{canonical module} for a local ring $(A,\frakm)$ is a finitely generated $A$-module~$\omega_A$ with
\[
\Hom_A(\omega_A,\,E)\ \cong\ H^{\dim A}_\frakm(A),
\]
where $E$ is the injective hull of the residue field $A/\frakm$ in the category of $A$-modules. The canonical module of $A$---when it exists---is unique up to isomorphism. Suppose $A$ is the homomorphic image of a Gorenstein local ring $R$. Then
\[
\Ext_R^{\dim R-\dim A}(A,\,R)
\]
is an $A$-module satisfying the Serre condition $S_2$, and is a canonical module for $A$.

A local ring $A$ is said to be \emph{quasi-Gorenstein} if it is the homomorphic image of a Gorenstein local ring, and $\omega_A$ is isomorphic to $A$. Using Andr\'e's Theorem~\cite{Andre}, one obtains:

\begin{theorem}
\label{theorem:quasi:gorenstein}
Let $(A,\frakm)$ be a local ring that is a homomorphic image of a Gorenstein local ring. Set $n\colonequals \dim A$. Then the natural map $H^n(\frakm;\,\omega_A)\to H_\frakm^n(\omega_A)$ is nonzero. In particular, if $A$ is quasi-Gorenstein, then the natural map $H^n(\frakm;\,A)\to H^n_\frakm(A)$ is nonzero.
\end{theorem}

\begin{proof}
Since the direct summand conjecture is true by \cite{Andre}, the local ring $A$ satisfies the canonical element property, see~\cite[Theorem~2.8]{Hochster:JALG}. By \cite[Theorem~4.3]{Hochster:JALG}, this is equivalent to the map $\Ext_A^n(A/\frakm,\,\omega_A)\to H_\frakm^n(\omega_A)$ being nonzero. This map factors as
\[
\Ext_A^n(A/\frakm,\,\omega_A)\to H^n(\frakm;\,\omega_A)\to H_\frakm^n(\omega_A),
\]
implying that the map $H^n(\frakm;\,\omega_A)\to H_\frakm^n(\omega_A)$ is nonzero.
\end{proof}

The next theorem records an interesting case where we have substantial information on the kernel of the map $H^n(\frakm;\,A)\to H^n_\frakm(A)$.

\begin{theorem}
\label{theorem:buchsbaum}
Let $A$ be a standard graded ring that is finitely generated over a field $k\colonequals A_0$. Set $n\colonequals\dim A$ and $e\colonequals\edim A$, and let $\frakm$ denote the homogeneous maximal ideal of $A$. Suppose there exists an integer $d$ such that for each integer $j$ with $j<n$, one has
\[
H_\frakm^j(A)\ =\ {[H_\frakm^j(A)]}_d.
\]
Set $s_j\colonequals \rank {[H_\frakm^j(A)]}_d$. Then:
\begin{enumerate}[\,\rm(1)]
\item The kernel of the natural map $\phi^n_\frakm\colon H^n(\frakm;\,A)\to H_\frakm^n(A)$ has Hilbert series
\[
\sum_i \rank{\big[\ker\phi^n_\frakm\big]}_i\,T^i
\ =\
s_{n-1}\binom{e}{1}T^{d-1} + s_{n-2}\binom{e}{2}T^{d-2} + \cdots + s_0\binom{e}{n}T^{d-n}.
\]
\item If $\fraka$ is an $\frakm$-primary ideal that is minimally generated by $r$ homogeneous elements, each of degree $t$, then the kernel of $H^n(\fraka;\,A)\to H_\frakm^n(A)$ has Hilbert series
\[
s_{n-1}\binom{r}{1}T^{d-t} + s_{n-2}\binom{r}{2}T^{d-2t} + \cdots + s_0\binom{r}{n}T^{d-nt}.
\]
\end{enumerate}
\end{theorem}

\begin{proof}
It suffices to prove the second assertion. Fix minimal generators $\bsz$ of $\fraka$. The Koszul complex $K^\bullet\colonequals K^\bullet(\bsz;\,A)$ takes the form
\[
\CD
0 @>>> A @>>> {A(t)}^{\binom{r}{1}} @>>> {A(2t)}^{\binom{r}{2}} @>>> \cdots @>>> {A(rt)}^{\binom{r}{r}} @>>> 0,
\endCD
\]
and lives in cohomology degree $0,1,\dots,r$.

Let $\omega_A^\bullet$ be the graded normalized dualizing complex of $A$, in which case $\omega_A=H^{-n}(\omega_A^\bullet)$ is the graded canonical module of $A$. Set $(-)^\vee=\Hom_A(-,\,{}^*E)$, where ${}^*E$ is the injective hull of $A/\frakm$ in the category of graded $A$-modules. Applying $R\Hom_A(K^\bullet,\,-)$ to the triangle
\[
\CD
\omega_A[n] @>>> \omega_A^\bullet @>>> \tau_{>-n}\omega_A^\bullet @>{+1}>>
\endCD
\]
gives
\[
\CD
R\Hom_A(K^\bullet,\,\omega_A[n]) @>>> R\Hom_A(K^\bullet,\,\omega_A^\bullet) @>>> R\Hom_A(K^\bullet,\,\tau_{>-n}\omega_A^\bullet) @>{+1}>>.
\endCD
\]
Since the complex $K^\bullet$ has Artinian cohomology, applying the functor $(-)^\vee$ gives
\[
\CD
{R\Hom_A(K^\bullet,\,\tau_{>-n}\omega_A^\bullet)}^\vee @>>> K^\bullet @>>> {R\Hom_A(K^\bullet,\,\omega_A[n])}^\vee @>{+1}>>.
\endCD
\]
This induces the exact sequence
\[
\CD
0\ @>>> H^n({R\Hom_A(K^\bullet,\,\tau_{>-n}\omega_A^\bullet)}^\vee) @>>> H^n(K^\bullet) @>>> H^n({R\Hom_A(K^\bullet,\,\omega_A[n])}^\vee),
\endCD
\]
where the zero on the left is because $R\Hom_A(K^\bullet,\,\omega_A[n])^\vee$ lives in cohomological degrees~$n, n+1,\dots, n+r$. The module $H^n({R\Hom_A(K^\bullet,\,\omega_A[n])}^\vee)$ is the kernel of the map
\[
\CD
H_\frakm^n(A)@>>> {H_\frakm^n(A)(t)}^{\binom{r}{1}}
\endCD
\]
given by multiplication by $z_i$ in the $i$-th coordinate. It follows that
\[
H^n({R\Hom_A(K^\bullet,\,\omega_A[n])}^\vee)\ =\ 0:_{H_\frakm^n(A)}\fraka.
\]
The map $H^n(K^\bullet)\to H^n(R\Hom_A(K^\bullet,\,\omega_A[n])^\vee)$ may be identified naturally with
\[
H^n(\fraka;\,A)\to 0:_{H_\frakm^n(A)}\fraka,
\]
which has the same kernel as $H^n(\fraka;\,A)\to H_\frakm^n(A)$ since $H^n(\fraka;\,A)$ is annihilated by~$\fraka$. Summarizing, one has
\[
H^n(R\Hom_A(K^\bullet,\,\tau_{>-n}\omega_A^\bullet)^\vee)\ =\ \ker\big(H^n(\fraka;\,A)\to H_\frakm^n(A)\big).
\]

The hypothesis that $H_\frakm^j(A)={[H_\frakm^j(A)]}_d$ for each integer $j$ with $j<n$ implies that~$R$ is Buchsbaum, see \cite[Theorem~3.1]{Schenzel}. Acknowledging the abuse of notation, we reuse the symbol $k$ below for the residue field $A/\frakm$. By \cite[Theorem~2.3]{Schenzel}, $\tau_{>-n}\omega_A^\bullet$ is quasi-isomorphic to the complex
\[
\CD
0 @>>> k^{s_{n-1}}(d) @>>> \cdots @>>> k^{s_1}(d) @>>> k^{s_0}(d) @>>> 0
\endCD
\]
of graded $k$-vector spaces, each in degree $-d$, with zero differentials; note that $k^{s_j}$ in the complex above has cohomology degree $-j$. Using this representative of $\tau_{>-n}\omega_A^\bullet$ to compute~$R\Hom_A(K^\bullet,\,\tau_{>-n}\omega_A^\bullet)^\vee$, the corresponding double complex takes the form:
\[
\CD
@. @AAA @AAA @. @AAA\\
0 @>>> {k(2t-d)}^{\binom{r}{2}s_0} @>>> {k(2t-d)}^{\binom{r}{2}s_1} @>>> \cdots @>>> {k(2t-d)}^{\binom{r}{2}s_{n-1}} @>>> 0\\
@. @AAA @AAA @. @AAA\\
0 @>>> {k(t-d)}^{\binom{r}{1}s_0} @>>> {k(t-d)}^{\binom{r}{1}s_1} @>>> \cdots @>>> {k(t-d)}^{\binom{r}{1}s_{n-1}} @>>> 0\\
@. @AAA @AAA @. @AAA\\
0 @>>> {k(-d)}^{s_0} @>>> {k(-d)}^{s_1} @>>> \cdots @>>> {k(-d)}^{s_{n-1}} @>>> 0
\endCD
\]
Note that all differentials are zero: the horizontal ones in light of the chosen representative for $\tau_{>-n}\omega_A^\bullet$, and the vertical ones since each column is a Koszul complex of the form~$K^\bullet(\fraka;\,k^{s_j})(-d)$. The bottom row sits in cohomology degree $0,1,\dots,n-1$ as it is the graded dual of $\tau_{>-n}\omega_A^\bullet$. Taking the total complex, one obtains
\[
H^n(R\Hom_A(K^\bullet,\,\tau_{>-n}\omega_A^\bullet)^\vee)
\ =\ 
k(t-d)^{\binom{r}{1}s_{n-1}}\oplus k(2t-d)^{\binom{r}{2}s_{n-2}}\oplus\cdots\oplus k(nt-d)^{\binom{r}{n}s_0}.\qedhere
\]
\end{proof}

We illustrate the preceding results with a number of examples, beginning with an elementary example taken from ~\cite[Remark~3.9]{Dutta:LMS}:

\begin{example}
Let $k$ be a field. Set $A\colonequals k[x,y]/(x^2, xy)$. Then $H^0_\frakm(A)=xA$, which is a rank~$1$ vector space, concentrated in degree~$1$. By Theorem~\ref{theorem:buchsbaum}, $\ker\big(H^1(\frakm;\,A)\to H_\frakm^1(A)\big)$ has rank $2$, and is concentrated in degree~$0$. This is confirmed by examining $K^\bullet(\frakm;\,A)$, i.e.,
\[
\CD
0 @>>> A @>{\begin{pmatrix}x\\ y\end{pmatrix}}>> A^2(1) @>{\begin{pmatrix}y&-x\end{pmatrix}}>> A(2) @>>> 0,
\endCD
\]
to observe that $H^1(\frakm;\,A)\cong k^2$, with the generators corresponding to
\[
\begin{pmatrix}0\\ x\end{pmatrix},\quad
\begin{pmatrix}0\\ y\end{pmatrix},
\]
and that each of these maps to zero under $H^1(\frakm;\,A)\to H^1_\frakm(A)$.

For integers $t\ge2$, Theorem~\ref{theorem:buchsbaum} says that the kernel of $H^1(y^t;\,A)\to H_\frakm^1(A)$ has Hilbert series~$T^{1-t}$. Indeed, this map is
\[
A(t)/Ay^t\to A_y/A,
\]
with the kernel being the rank $1$ vector space generated by $x\ \in\ {[A(t)/Ay^t]}_{1-t}$.

Lastly, note that the image of $H^1(y^t;\,A)\to H_\frakm^1(A)$ has Hilbert series
\[
T^{-1}+T^{-2}+T^{-3}+\dots+T^{-t}
\]
which agrees with the Hilbert series of $H_\frakm^1(A)$ as $t\to\infty$.
\end{example}

\begin{example}
Let $A$ be as in Theorem~\ref{theorem:segre} where $f(\bsx)$ has degree $b$. Then the ring $A$ is not Cohen-Macaulay, and the K\"unneth formula gives
\[
H_\frakm^j(A)\ =\begin{cases}
k &\text{ if $j=b-1$,}\\
0 &\text{ if $j\neq b-1,\ \ b+c-2$}.
\end{cases}
\]
The map $H^{b+c-2}(\frakm;\,A)\to H_\frakm^{b+c-2}(A)$ is zero by Theorem~\ref{theorem:segre}, while Theorem~\ref{theorem:buchsbaum} says that its kernel, i.e., $H^{b+c-2}(\frakm;\,A)$, has Hilbert series
\[
\binom{bc}{c-1}T^{-(c-1)}.
\]
\end{example}

\begin{example}
Let $k$ be a field; take polynomial rings $k[x_1,\dots,x_b]$ and~$C\colonequals k[y_1,\dots,y_c]$, where $b\ge 3$ and $c\ge 3$. Let $f(\bsx)$ and $g(\bsy)$ be homogeneous polynomials of degrees $b$ and~$c$ respectively, such that the hypersurfaces
\[
B\colonequals k[x_1,\dots,x_b]/(f(\bsx))\quad\text{ and }\quad C\colonequals k[y_1,\dots,y_c]/(g(\bsy))
\]
are normal. The ring $A\colonequals B\,\#\,C$ is normal, and of dimension $b+c-3$; let $\frakm$ denote the homogeneous maximal ideal of $A$. Then
\[
H_\frakm^j(A)\ =\begin{cases}
0 &\text{ if $j\neq b-1,\ \ c-1$, \ or $ b+c-3$,}\\
k &\text{ if $b\neq c$, and $j$ equals either $b-1$ or $c-1$,} \\
k^2 &\text{ if $b=c$ and $j=b-1$.}
\end{cases}
\]
Since $B$ and $C$ are each Gorenstein with $a$-invariant $0$, the ring $A$ is quasi-Gorenstein by~\cite[Theorem~4.3.1]{GW}. Hence Theorem~\ref{theorem:quasi:gorenstein} says that $H^{b+c-3}(\frakm;\,A)\to H_\frakm^{b+c-3}(A)$ is nonzero, while Theorem~\ref{theorem:buchsbaum} implies that the kernel has Hilbert series
\[
\binom{bc}{c-2}T^{-(c-2)} + \binom{bc}{b-2}T^{-(b-2)}.
\]
\end{example}

We conclude this section with an example where $A$ is not Cohen-Macaulay or quasi-Gorenstein, but the map $H^{\dim A}(\frakm;\,A)\to H^{\dim A}_\frakm(A)$ is nonzero:

\begin{example}
Let $k$ be a field, and set $A\colonequals k[x,y,z]/(xz, y^2, yz, z^2)$. Then $\dim A=1$, and~$x$ is a homogeneous parameter. The Koszul complex $K^\bullet(\frakm;\,A)$ is
\[
\CD
0 @>>> A @>{\begin{pmatrix}x\\ y\\ z\end{pmatrix}}>> A^3(1) @>{\begin{pmatrix}0&-z&y\\ -z&0&x\\ -y&x&0 \end{pmatrix}}>> A^3(2)
@>{\begin{pmatrix}x&-y&z\end{pmatrix}}>> A(3) @>>> 0,
\endCD
\]
from which it follows that $H^1(\frakm;\,A)\cong k^4$, with the four generators corresponding to
\[
\begin{pmatrix}z\\ 0\\ 0\end{pmatrix},\quad
\begin{pmatrix}0\\ z\\ 0\end{pmatrix},\quad
\begin{pmatrix}0\\ 0\\ z\end{pmatrix},\quad
\begin{pmatrix}y\\ 0\\ 0\end{pmatrix}.
\]
The first three generators map to zero under $H^1(\frakm;\,A)\to H^1_\frakm(A)$, whereas the fourth maps to the nonzero element of $H^1_\frakm(A)$ represented by the \v Cech cocycle
\[
\left(\frac{y}{x},0,0\right)\ \in\ A_x\oplus A_y\oplus A_z.
\]
The ring $A$ is not $S_2$, and hence not quasi-Gorenstein.

Note that $H^0_\frakm(A)=zA$ is a rank~$1$ vector space concentrated in degree~$1$, so Theorem~\ref{theorem:buchsbaum} also confirms that $\ker\big(H^1(\frakm;\,A)\to H_\frakm^1(A)\big)$ has rank $3$, and is concentrated in degree~$0$.

For integers $t\ge2$, Theorem~\ref{theorem:buchsbaum} says that the kernel of $H^1(x^t;\,A)\to H_\frakm^1(A)$ has Hilbert series~$T^{1-t}$. Indeed, $H^1(x^t;\,A)$ has Hilbert series
\[
1 + 2T^{-1} + 2T^{-2} + \dots + 2T^{2-t} + 3T^{1-t} + T^{-t},
\]
while $H_\frakm^1(A)$ has Hilbert series $1 + 2T^{-1} + 2T^{-2} + 2T^{-3}+ \cdots$.
\end{example}

\section{Stanley-Reisner rings}
\label{section:stanley:reisner}

Theorem~\ref{theorem:buchsbaum} is perhaps best viewed in the broader context of filtering a local cohomology module $H^n_\frakm(A)$---that is typically not finitely generated---using natural finitely generated submodules such as the images of Koszul cohomology or $\Ext$ modules, themes that are pursued at length in \cite{BBLSZ} and \cite{SW}. This turns out to be fascinating even in the context of Stanley-Reisner rings; in this section, we examine the implications of Theorem~\ref{theorem:buchsbaum} in some extensively studied examples such as triangulations of the torus and of the real projective plane. First, some generalities:

Let~$\Delta$ be a simplicial complex with vertices $1,\dots,e$. For $k$ a field, consider the polynomial ring~$k[x_1,\dots,x_e]$ and the ideal $\fraka$ generated by the square-free monomials
\[
x_{i_1}\cdots x_{i_r}
\]
such that $\{i_1,\dots, i_r\}$ is not a face of $\Delta$. The \emph{Stanley-Reisner ring} of $\Delta$ over $k$ is the ring
\[
A\colonequals k[x_1,\dots,x_e]/\fraka.
\]
The ring $A$ has a $\ZZ^e$-grading with $\deg x_i$ being the $i$-th unit vector; this induces a grading on the \v Cech complex $C^\bullet(\bsx;\,A)$. The module
\[
A_{x_{i_1}\cdots x_{i_r}}
\]
is nonzero precisely if $x_{i_1}\cdots x_{i_r}\notin\fraka$, equivalently $\{i_1,\dots,i_r\}\in\Delta$. Hence the graded strand of~$C^\bullet(\bsx;\,A)$ in degree $\bzero\colonequals(0,\dots,0)$ is the complex that computes the reduced simplicial cohomology $\tilde{H}^\bullet(\Delta;\,k)$, with the indices shifted by one, so
\[
{[H^j_\frakm(A)]}_\bzero\ \cong\ \tilde{H}^{j-1}(\Delta;\,k)\qquad\text{ for }j\ge0\,.
\]

Theorem~\ref{theorem:buchsbaum} yields the following corollary for Stanley-Reisner rings:

\begin{corollary}
\label{corollary:stanley:reisner}
Let $A\colonequals k[x_1,\dots,x_e]/\fraka$ be an equidimensional Stanley-Reisner ring over a field $k$, such that~$A_\frakp$ is Cohen-Macaulay for each $\frakp\in\Spec A\smallsetminus\{\frakm\}$. Set $n\colonequals\dim A$. For $t$ a positive integer, set $\frakm^{[t]}$ to be the ideal of $A$ generated by the images of $x_1^t,\dots,x_e^t$.

Then the kernel of the natural map $H^n(\frakm^{[t]};\,A)\to H_\frakm^n(A)$ has Hilbert series
\[
s_{n-1}\binom{e}{1}T^{-t} + s_{n-2}\binom{e}{2}T^{-2t} + \cdots + s_0\binom{e}{n}T^{-nt},
\]
where $s_j\colonequals \rank \tilde{H}^{j-1}(\Delta;\,k)$, with $\Delta$ denoting the underlying simplicial complex.
\end{corollary}

\begin{proof}
The hypotheses that $A$ is equidimensional and that $A_\frakp$ is Cohen-Macaulay for $\frakp\neq\frakm$ imply that each local cohomology module $H_\frakm^j(A)$, for $j<n$, has finite length. Let $F$ denote the $k$-algebra endomorphism of $A$ with $x_i\mapsto x_i^2$ for each $i$. Then $F$ is a pure endomorphism by \cite[Example~2.2]{SW}, so the induced map
\[
\tilde{F}\colon H_\frakm^j(A) \to H_\frakm^j(A)
\]
is injective for each $j$. Using the $\ZZ^e$-grading from the preceding discussion, $\tilde{F}$ restricts to
an injective map
\[
{[H_\frakm^j(A)]}_\bsi\to{[H_\frakm^j(A)]}_{2\bsi}
\]
for each $\bsi\in\ZZ^e$. But~$H_\frakm^j(A)$ has finite length for $j<n$, so
\[
H_\frakm^j(A)\ =\ {[H_\frakm^j(A)]}_\bzero
\]
for each $j<n$. The result now follows by Theorem~\ref{theorem:buchsbaum}.
\end{proof}

We begin by using Corollary~\ref{corollary:stanley:reisner} to shed light on Example~\ref{example:real}:

\begin{example}
\label{example:two:lines}
Consider the simplicial complex corresponding to two disjoint line segments; the corresponding Stanley-Reisner ring $A$, over a field $k$, is
\[
k[x_1,\,x_2,\,x_3,\,x_4]/(x_1x_3,\ x_1x_4,\ x_2x_3,\ x_2x_4)
\]
The ring $A$ has dimension $2$, and is not Cohen-Macaulay since
\[
{[H^1_\frakm(A)]}_\bzero\ \cong\ \tilde{H}^0(\Delta;\,k)\ \cong\ k.
\]
Since $A$ is equidimensional and Cohen-Macaulay on the punctured spectrum, Corollary~\ref{corollary:stanley:reisner} says that the kernel of $H^2(\frakm^{[t]};\,A)\to H_\frakm^2(A)$ has Hilbert series $4T^{-t}$ for each~$t\ge1$. This is consistent with the following table, where we record the rank of the vector spaces
\[
{[H^2(\frakm^{[t]};\,A)]}_j
\]
as computed by \emph{Macaulay2} \cite{GS}, in the case $k$ is the field of rational numbers. Entries that are $0$ are omitted. The last row records the rank of
\[
\left[\lim_{t\rightarrow\infty} H^2(\frakm^{[t]};\,A)\right]_j\ =\ \left[H_\frakm^2(A)\right]_j,
\]
which may be obtained using the exact sequence
\[
\CD
0 @>>> A @>>> A/(x_1,x_2)\oplus A/(x_3,x_3) @>>> A/\frakm @>>> 0
\endCD
\]
and the induced isomorphism $H_\frakm^2(A)\ \cong\ H_\frakm^2(A/(x_1,x_2))\oplus H_\frakm^2(A/(x_3,x_4)).$

\begin{table}[H]
\begin{center}
\begin{tabular}{|c||C{.6cm}|C{.6cm}|C{.6cm}|C{.6cm}|C{.6cm}|C{.6cm}|C{.6cm}|C{.6cm}|C{.6cm}|C{.6cm}|}
\hline
\backslashbox{$t$}{$j$} & $-1$ & $-2$ & $-3$ & $-4$ & $-5$ & $-6$ & $-7$ & $-8$ & $-9$ & $-10$\\
\hline
\hline
$1$ & $4$ & & & & & & & & & \\
\hline
$2$ & & $6$ & $4$ & & & & & & & \\
\hline
$3$ & & $2$ & $8$ & $6$ & $4$ & & & & & \\
\hline
$4$ & & $2$ & $4$ & $10$ & $8$ & $6$ & $4$ & & & \\
\hline
$5$ & & $2$ & $4$ & $6$ & $12$ & $10$ & $8$ & $6$ & $4$ & \\
\hline
$6$ & & $2$ & $4$ & $6$ & $8$ & $14$ & $12$ & $10$ & $8$ & $6$ \\
\hline
$7$ & & $2$ & $4$ & $6$ & $8$ & $10$ & $16$ & $14$ & $12$ & $10$ \\
\hline
$8$ & & $2$ & $4$ & $6$ & $8$ & $10$ & $12$ & $18$ & $16$ & $14$ \\
\hline
$9$ & & $2$ & $4$ & $6$ & $8$ & $10$ & $12$ & $14$ & $20$ & $18$ \\
\hline
$10$ & & $2$ & $4$ & $6$ & $8$ & $10$ & $12$ & $14$ & $16$ & $22$ \\
\hline
\hline
$\lim_{t\rightarrow\infty}$ & & $2$ & $4$ & $6$ & $8$ & $10$ & $12$ & $14$ & $16$ & $18$ \\
\hline
\end{tabular}
\end{center}
\end{table}
Note that for the ring $A\colonequals\RR[x,y,ix,iy]$ in Example~\ref{example:real}, the tensor product $A\otimes_\RR\CC$ is isomorphic to $\CC[x_1,x_2,x_3,x_4]/(x_1x_3,\ x_1x_4,\ x_2x_3,\ x_2x_4)$ since
\[
A\ \cong\ \RR[x,y,u,v]/(x^2+u^2,\ y^2+v^2,\ vx-uy,\ xy+uv),
\]
and hence $A\otimes_\RR\CC$ is isomorphic to $\CC[x,y,u,v]/\big((u-ix,\ v-iy)\cap(u+ix,\ v+iy)\big)$.
\end{example}

\begin{example}
Consider the triangulation of the torus below.
\begin{figure}[ht]
\[
\caption{A triangulation of the torus}
\psfrag{1}{\footnotesize$1$}
\psfrag{2}{\footnotesize$2$}
\psfrag{3}{\footnotesize$3$}
\psfrag{4}{\footnotesize$4$}
\psfrag{5}{\footnotesize$5$}
\psfrag{6}{\footnotesize$6$}
\psfrag{7}{\footnotesize$7$}
\psfrag{8}{\footnotesize$8$}
\psfrag{9}{\footnotesize$9$}
\includegraphics[width=3cm]{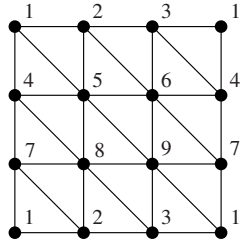}
\]
\end{figure}

The corresponding Stanley-Reisner ring $A$ is the homomorphic image of~$k[x_1,\dots,x_9]$ modulo the ideal generated by the monomials $x_1x_6$, $x_1x_8$, $x_2x_4$, $x_2x_9$, $x_3x_5$, $x_3x_7$, $x_4x_9$, $x_5x_7$, $x_6x_8$, $x_1x_2x_3$, $x_1x_4x_7$, $x_1x_5x_9$, $x_2x_5x_8$, $x_2x_6x_7$, $x_3x_4x_8$, $x_3x_6x_9$, $x_4x_5x_6$, and $x_7x_8x_9$. The ring $A$ has dimension $3$, and is not Cohen-Macaulay since
\[
{[H^2_\frakm(A)]}_\bzero\ \cong\ \tilde{H}^1(\Delta;\,k)\ \cong\ k^2.
\]
It is readily verified that $A$ is Cohen-Macaulay on the punctured spectrum. The canonical module of $A$ may be computed as
\[
\omega_A\ =\ \Ext^6_{k[\bsx]}(A,\,\omega_{k[\bsx]})\ \cong\ A,
\]
so the ring $A$ is quasi-Gorenstein. Hence the natural map $H^3(\frakm;\,A)\to H^3_\frakm(A)$ is nonzero, though not injective. Indeed, the module $H^3(\frakm;\,A)$ has Hilbert series $1+18T^{-1}$, while Corollary~\ref{corollary:stanley:reisner} implies that the kernel of $H^3(\frakm^{[t]};\,A)\to H_\frakm^3(A)$ has Hilbert series $18T^{-t}$ for each integer $t\ge1$. The following table records the rank of the vector spaces
\[
{[H^3(\frakm^{[t]};\,A)]}_j,
\]
as computed by \emph{Macaulay2} in the case $k$ is the field of rational numbers. The last row records the rank of
\[
\left[\lim_{t\rightarrow\infty} H^3(\frakm^{[t]};\,A)\right]_j\ =\ \left[H_\frakm^3(A)\right]_j,
\]
which may be obtained using the Hilbert series of $A$: since the ring $A$ is quasi-Gorenstein with $a$-invariant $0$, one has $\rank\left[H_\frakm^3(A)\right]_j=\rank {[A]}_{-j}$.

\begin{table}[H]
\begin{center}
\begin{tabular}{|c||C{.6cm}|C{.6cm}|C{.6cm}|C{.6cm}|C{.6cm}|C{.6cm}|C{.6cm}|C{.6cm}|C{.6cm}|C{.6cm}|C{.6cm}|}
\hline
\backslashbox{$t$}{$j$} & $0$ & $-1$ & $-2$ & $-3$ & $-4$ & $-5$ & $-6$ & $-7$ & $-8$ & $-9$ & $-10$\\
\hline
\hline
$1$ & $1$ & $18$ & & & & & & & & & \\
\hline
$2$ & $1$ & $9$ & $45$ & $18$ & & & & & & & \\
\hline
$3$ & $1$ & $9$ & $36$ & $90$ & $81$ & $54$ & $18$ & & & & \\
\hline
$4$ & $1$ & $9$ & $36$ & $81$ & $153$ & $162$ & $153$ & $108$ & $54$ & $18$ & \\
\hline
$5$ & $1$ & $9$ & $36$ & $81$ & $144$ & $234$ & $261$ & $270$ & $243$ & $180$ & $108$\\
\hline
$6$ & $1$ & $9$ & $36$ & $81$ & $144$ & $225$ & $333$ & $378$ & $405$ & $396$ & $351$ \\
\hline
$7$ & $1$ & $9$ & $36$ & $81$ & $144$ & $225$ & $324$ & $450$ & $513$ & $558$ & $567$ \\
\hline
$8$ & $1$ & $9$ & $36$ & $81$ & $144$ & $225$ & $324$ & $441$ & $585$ & $666$ & $729$ \\
\hline
$9$ & $1$ & $9$ & $36$ & $81$ & $144$ & $225$ & $324$ & $441$ & $576$ & $738$ & $837$ \\
\hline
$10$ & $1$ & $9$ & $36$ & $81$ & $144$ & $225$ & $324$ & $441$ & $576$ & $729$ & $909$ \\
\hline
\hline
$\lim_{t\rightarrow\infty}$ & $1$ & $9$ & $36$ & $81$ & $144$ & $225$ & $324$ & $441$ & $576$ & $729$ & $900$ \\
\hline
\end{tabular}
\end{center}
\end{table}
\end{example}

We conclude with an example where the cohomology is characteristic-dependent:

\begin{example}
Consider the simplicial complex corresponding to the triangulation of the real projective plane $\RR\PP^2$ in Figure~2.
\begin{figure}[ht]
\[
\caption{A triangulation of the real projective plane}
\psfrag{1}{\footnotesize$1$}
\psfrag{2}{\footnotesize$2$}
\psfrag{3}{\footnotesize$3$}
\psfrag{4}{\footnotesize$4$}
\psfrag{5}{\footnotesize$5$}
\psfrag{6}{\footnotesize$6$}
\includegraphics[width=3cm]{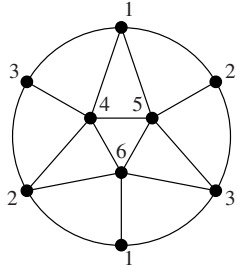}
\]
\end{figure}
The corresponding Stanley-Reisner ring $A$ is the homomorphic image of $k[x_1,\dots,x_6]$ modulo the ideal generated $x_1x_2x_3$, $x_1x_2x_4$, $x_1x_3x_5$, $x_1x_4x_6$, $x_1x_5x_6$, $x_2x_3x_6$, $x_2x_4x_5$, $x_2x_5x_6$, $x_3x_4x_5$, and $x_3x_4x_6$. Note that $\dim A=3$.

Suppose the field $k$ has characteristic other than $2$. Then $A$ is Cohen-Macaulay, see for example~\cite[page~180]{Hochster:simplicial}, so the natural map $H^3(\frakm;\,A)\to H^3_\frakm(A)$ is injective by Theorem~\ref{theorem:cohen:macaulay}. However, $A$ is not Gorenstein: the socle of the homomorphic image of $A$ modulo a system of parameters has rank $6$.

Next, suppose $k$ has characteristic $2$; then $A$ is not Cohen-Macaulay since
\[
{[H^2_\frakm(A)]}_\bzero\ \cong\ \tilde{H}^1(\Delta;\,k)\ \cong\ k.
\]
Indeed, in this case, $A$ has depth $2$, and the canonical module of $A$ may be computed as
\[
\omega_A\ =\ \Ext^3_{k[\bsx]}(A,\,\omega_{k[\bsx]})\ \cong\ A,
\]
so $A$ is quasi-Gorenstein. Hence the natural map $H^3(\frakm;\,A)\to H^3_\frakm(A)$ is nonzero, though not injective. The module $H^3(\frakm;\,A)$ has Hilbert series $1+6T^{-1}$, and Corollary~\ref{corollary:stanley:reisner} implies that the kernel of $H^3(\frakm^{[t]};\,A)\to H_\frakm^3(A)$ has Hilbert series $6T^{-t}$ for each~$t\ge1$. The ranks of the vector spaces
\[
{[H^3(\frakm^{[t]};\,A)]}_j
\]
are recorded in the next table, with the last row computed as in the preceding example, using that $A$ is quasi-Gorenstein, with $a$-invariant $0$.

\begin{table}[H]
\begin{center}
\begin{tabular}{|c||C{.6cm}|C{.6cm}|C{.6cm}|C{.6cm}|C{.6cm}|C{.6cm}|C{.6cm}|C{.6cm}|C{.6cm}|C{.6cm}|C{.6cm}|}
\hline
\backslashbox{$t$}{$j$} & $0$ & $-1$ & $-2$ & $-3$ & $-4$ & $-5$ & $-6$ & $-7$ & $-8$ & $-9$ & $-10$\\
\hline
\hline
$1$ & $1$ & $6$ & & & & & & & & & \\
\hline
$2$ & $1$ & $6$ & $21$ & $10$ & & & & & & & \\
\hline
$3$ & $1$ & $6$ & $21$ & $46$ & $45$ & $30$ & $10$ & & & & \\
\hline
$4$ & $1$ & $6$ & $21$ & $46$ & $81$ & $90$ & $85$ & $60$ & $30$ & $10$ & \\
\hline
$5$ & $1$ & $6$ & $21$ & $46$ & $81$ & $126$ & $145$ & $150$ & $135$ & $100$ & $60$\\
\hline
$6$ & $1$ & $6$ & $21$ & $46$ & $81$ & $126$ & $181$ & $210$ & $225$ & $220$ & $195$\\
\hline
$7$ & $1$ & $6$ & $21$ & $46$ & $81$ & $126$ & $181$ & $246$ & $285$ & $310$ & $315$ \\
\hline
$8$ & $1$ & $6$ & $21$ & $46$ & $81$ & $126$ & $181$ & $246$ & $321$ & $370$ & $405$ \\
\hline
$9$ & $1$ & $6$ & $21$ & $46$ & $81$ & $126$ & $181$ & $246$ & $321$ & $406$ & $465$ \\
\hline
$10$ & $1$ & $6$ & $21$ & $46$ & $81$ & $126$ & $181$ & $246$ & $321$ & $406$ & $501$ \\
\hline
\hline
$\lim_{t\rightarrow\infty}$ & $1$ & $6$ & $21$ & $46$ & $81$ & $126$ & $181$ & $246$ & $321$ & $406$ & $501$ \\
\hline
\end{tabular}
\end{center}
\end{table}
\end{example}


\end{document}